\documentclass[12pt]{amsart}
\usepackage{latexsym,amsmath,amsfonts,amscd,amssymb}
\newtheorem{theorem}{\bf Theorem}[section]
\newtheorem{lemma}[theorem]{\bf Lemma}
\newtheorem{remark}[theorem]{\bf Remark}

\newtheorem{proposition}[theorem]{\bf Proposition}
\newtheorem{corollary}[theorem]{\bf Corollary}

\newtheorem{example}[theorem]{\bf Example}

\newcommand{\be}{\begin{equation}}

\newfont{\bfc}{cmbsy10 scaled 1200}  
\newfont{\dr}{msbm10 scaled \magstep1}  
\newfont{\sdr}{msbm8}  
\newfont{\gl}{eufm10 scaled \magstep1}  

\DeclareFontFamily{OT1}{rsfs}{}
\DeclareFontShape{OT1}{rsfs}{n}{it}{<->rsfs10}{}
\DeclareMathAlphabet{\curly}{OT1}{rsfs}{n}{it}

 \newcommand{\pf}{{\em Proof}. }

 \newcommand{\PP}{{\mathbb P}}

 \newcommand{\cO}{{\mathcal O}}

\newcommand{\qu}{/\kern-.7ex/}
\newcommand{\exh}{\to\kern-1.8ex\to}

\newcommand{\lra}{\longrightarrow}

\newcommand{\Cl}{\operatorname{Cliff}(C)}
\newcommand{\Cln}{\operatorname{Cliff}_n(C)}
\newcommand{\Clt}{\operatorname{Cliff}_2(C)}
\newcommand{\Clo}{\operatorname{Cliff}_1(C)}

\newcommand{\Hom}{\operatorname{Hom}}
\newcommand{\Ext}{\operatorname{Ext}}

\begin{document}
\title[Stable bundles of rank 2]{Stable bundles of rank 2 with\\
4 sections}
\thanks{All authors are members of the research group VBAC
(Vector Bundles on Algebraic Curves). The third author would like to thank California State University Channel Islands, where parts of this research were carried out, and the Isaac Newton Institute, where the work was completed during the Moduli Spaces programme. We would also like to thank the EPSRC (Grant GR/R32093) for supporting the first author on a visit to Liverpool and the European Commission for supporting the second author as a Marie Curie Research Fellow in Liverpool for a period during which work on the original version of this paper was carried out}
\subjclass[2000]{14H60}
\date{\today}
\keywords{Algebraic curves, stable bundles, Brill-Noether loci,  coherent
systems, Clifford index}
\author{I.~Grzegorczyk}
 \address{Department of Mathematics\\University of California Channel Islands\\
 One University Drive\\Camarillo\\CA 93012\\USA}
 \email{Ivona.grze@csuci.edu}
\author{V.~Mercat}
  \address{5 rue Delouvain, 75019 Paris, France}
  \email{vipani.mm@math.gmail.com}
\author{P.~E.~Newstead}
  \address{Department of Mathematical Sciences \\
  University of Liverpool \\ Peach Street \\
  Liverpool L69 7ZL \\ UK}
  \email{newstead@liverpool.ac.uk}
\begin{abstract} This paper contains results on stable bundles of rank $2$ with space of sections of dimension $4$ on a smooth irreducible projective algebraic curve $C$. There is a known lower bound on the degree for the existence of such bundles; the main result of the paper is a geometric criterion for this bound to be attained. For a general curve $C$ of genus $10$, we show that the bound cannot be attained, but that there exist Petri curves of this genus for which the bound is sharp. We interpret the main results for various curves and in terms of Clifford indices and coherent systems. The results can also be expressed in terms of Koszul cohomology and the methods provide a useful tool for the study of the geometry of the moduli space of curves.

\end{abstract}

\maketitle
\section{Introduction}\label{section:intro}
This paper began life some years ago with a proof that, on a general curve of genus $10$, there does not exist a stable bundle of rank $2$ and degree $12$ with $4$ independent sections, but that such bundles do exist on certain Petri curves of genus $10$. The motivation for this came partly from a paper of C. Voisin \cite{Voi} concerning bundles of rank $2$ with canonical determinant and partly from results involving the Clifford index contained in \cite{M4}. In fact, the classical Clifford's theorem for line bundles has been extended to semistable bundles of higher rank by  G. Xiao \cite[p. 477]{X} (see also \cite[Theorem 2.1]{BGN}) as follows: if $E$
is a semistable vector bundle of rank $n$ and slope $\mu  = \frac{d}n$ on a smooth projective curve $C$ of genus $g$ such that $0 \leq \mu \leq 2g - 2$, then
$h^0(E) \leq \frac{d}{2} + n$. This result was improved by R.~Re \cite{Re} and further refined by the second author \cite{M4} to a version which takes into account the Clifford index $\Cl$ of $C$.  In \cite{M4}, it is conjectured that, for $ g\ge4$ and $E$ semistable, 
\begin{itemize}
\item  if $\Cl+2 \leq \mu\le  2g -4 - {\Cl}$, then $h^0 (E) \leq \frac{d-{\Cl}n}2+n $;
\item if $1 \leq \mu \leq \Cl + 2$, then $h^0(E) \leq \frac{1}{\Cl +1}(d-n) + n$.
\end{itemize}
Note that, for a general curve of genus $10$, the Clifford index is $4$, so, when $n=2$, the main part of the conjecture becomes
$$h^0 (E) \leq \frac{d}2-2 \mbox{ for }6 \leq \mu\le  12.$$
In particular bundles of degree $12$ with $h^0=4$ represent a ``corner point'' for these inequalities. 

The result mentioned at the beginning of the previous paragraph was never published and, as far as we are aware, has not been published since except in a brief form in the survey article \cite{n}. However, there have been several developments that add to the interest of the result and also place it in a more general context. The most important of these has been the increasing understanding that higher rank Brill-Noether theory can provide insights into the geometry of the moduli space of curves through the medium of Koszul cohomology (see, for example, \cite{an, an2}); in this respect the present paper adds to the evidence that higher rank Brill-Noether theory is not a consequence of the classical theory for rank $1$. Related to this are the broadening of higher rank Brill-Noether theory into the study of coherent systems on curves (in fact the form of the result in \cite{n} is in this context), the introduction of higher rank Clifford indices (see below for details) leading to a reformulation of the conjecture of \cite{M4} and the fact that, very recently, a counter-example to the conjecture for $n=3$ has been noted \cite{mln}. For $n=2$, we do not construct any counter-examples here, but our results do suggest where one might look for them; for recent developments including the construction of such counter-examples, see section \ref{ps}.

A final reason for interest in the construction of bundles of rank $2$ with many sections is that new estimates for the dimensions of Brill-Noether loci with fixed determinant have been obtained, also very recently, by B. Osserman \cite{Oss}.

In order to state our results, we need some definitions. Let $C$ be a smooth irreducible projective algebraic curve of genus $g\ge2$ defined over the complex numbers, and let $M(n,d)$ (resp. $\widetilde{M}(n,d)$) denote the moduli space of stable vector bundles (resp. S-equivalence classes of semistable vector bundles) of rank $n$ and degree $d$ over $C$. The {\em Brill-Noether locus}
$B(n,d,k)$ is the subvariety of $M(n,d)$ defined by
$$B(n,d,k)=\{E\in M(n,d)\;|\;h^0(E)\ge k\}.$$
Corresponding subvarieties $\widetilde{B}(n,d,k)$ of  $\widetilde{M}(n,d)$ can be defined similarly; note that $B(n,d,k)=\widetilde{B}(n,d,k)$ when $\gcd(n,d)=1$. The most fundamental question concerning these loci is to determine when they are non-empty (see \cite{GT} for a survey of results on this and other questions in the Brill-Noether theory of vector bundles). In particular, the {\it expected dimension} of $B(n,d,k)$ is given by the {\it Brill-Noether number}
$$\beta(n,d,k):=n^2(g-1)+1-k(k-d+n(g-1)).$$
Every component of $B(n,d,k)$ has dimension $\ge\beta(n,d,k)$, but, even on a general curve, it is not true that $\beta(n,d,k)\ge0$ implies the non-emptiness of $B(n,d,k)$, nor is the converse of this result true.The purpose of this paper is to contribute to the study of the question of non-emptiness in the first case for which the answer is not completely known, namely $n=2$, $k=4$. 

For general curves, M. Teixidor i Bigas has obtained non-emptiness results using degeneration methods \cite{T1} and L. Brambila-Paz has obtained both emptiness and non-emptiness results by other methods \cite{BP}. However we shall work in a wider context. For this, we need some definitions from \cite{ln}. For any vector bundle $E$ of rank $n$ and degree $d$ on $C$, we define
$$
\gamma(E) := \frac{1}{n} \left(d - 2(h^0(E) -n)\right) = \mu(E) -2\frac{h^0(E)}{n} + 2.
$$
If $C$ has genus $g \geq 4$, we then define, for any positive integer $n$,
$$
\Cln:= \min_{E} \left\{ \gamma(E) \;\left| 
\begin{array}{c} E \;\mbox{semistable of rank}\; n \\
h^0(E) \geq 2n,\; \mu(E) \leq g-1
\end{array} \right. \right\}
$$
(this invariant is denoted in \cite{ln} by $\gamma_n'$). Note that $\Clo = \Cl$ is the usual Clifford index of the curve $C$. Moreover, as observed in \cite[Proposition 3.3 and Conjecture 9.3]{ln}, the conjecture of \cite{M4} can be restated in a slightly weaker form as

\noindent{\bf Conjecture.} $\Cln=\Cl$. 

Next, the {\em gonality sequence} $\{d_r\}$ of $C$ is defined by
$$
d_r := \min \{ \deg L \;|\; L \; \mbox{a line bundle on} \; C \; \mbox{with} \; h^0(L) \geq r +1\}.
$$ 
We say that $d_r$ {\em computes} $\Cl$ if $d_r\le g-1$ and $\Cl=d_r-2r$. By classical Brill-Noether theory, we have always
\begin{equation}\label{eqn14}
d_r\le g+r-\left[\frac{g}{r+1}\right].
\end{equation}
This formula is valid for all $g$.

We need also the concept of a {\em Petri} curve. In general, for any line bundle $L$ on $C$, we consider the {\it Petri map}
$$H^0(L)\otimes H^0(K\otimes L^*)\lra H^0(K)$$
given by multiplication of sections. A curve $C$ is said to be a {\it Petri curve} if the Petri map is injective for every line bundle $L$ on $C $. We shall need the following facts concerning Petri curves:
\begin{itemize}
\item the general curve of any genus is Petri;
\item there exist Petri curves lying on K3 surfaces (see \cite{Laz});
\item on a Petri curve, the inequalities of \eqref{eqn14} are all equalities. 
\end{itemize}

We now summarise the results of the paper. In section \ref{rank2}, we obtain results on semistable and stable bundles of rank $2$ with $h^0\ge4$, especially those of degree $\le2d_1$ (note that such a bundle always has degree $\ge\min\{2d_1,d_4\}$ (see Proposition \ref{prop2} below and compare \cite[Theorem 5.2]{ln})). In particular we determine the emptiness or non-emptiness of $B(2,d,4)$ and $\widetilde{B}(2,d,4)$ in many cases (Proposition \ref{prop1} and Remark \ref{r8}). We also discuss these results for various different curves (Examples \ref{ex1}--\ref{ex3}). 

In section \ref{d4}, we suppose $d_4\le2d_1$ and concentrate on stable bundles of this type which have the minimum possible degree $d_4$. This leads to our first main theorem:
\newpage
\noindent{\bf Theorem \ref{main}.} {\em 
Suppose $d_4\le2d_1$. Then
\begin{itemize}
\item[(i)] if $B(2,d_4,4)\ne\emptyset$, there exists a line bundle $M\in B(1,d_4,5)$ such that $\phi_M(C)$ is contained in a quadric;
\item[(ii)] if $d_4<2d_1$ and there exists $M\in B(1,d_4,5)$ such that $\phi_M(C)$ is contained in a quadric, then $B(2,d_4,4)\ne\emptyset$. 
\end{itemize}}
\noindent (Here, for any generated line bundle $M$, $\phi_M$ denotes the morphism $C\to{\mathbb P}(H^0(M)^*)$ defined by evaluation of sections of $M$.) We relate this result to possible counter-examples to the conjecture $\Clt=\Cl$ (Remark \ref{r3}) and, for Petri curves, to the possible existence of non-empty Brill-Noether loci for which$$\beta(2,d,4):=4g-3-4(4-d+2g-2)$$ is negative (Remark \ref{r4}). We obtain also an upper bound for $\dim B(2,d_4,4)$ for a general curve of genus $g\ge10$ (Proposition \ref{newprop}). 

In section \ref{g=10}, we prove the main result of the original version of this paper, namely

\noindent{\bf Theorem \ref{main2}.} (i) {\em Let $C$ be a general curve of genus $10$. Then $$B(2,d,4)\ne\emptyset\Longleftrightarrow d\ge13\Longleftrightarrow \beta(2,d,4)\ge0.$$}
\hspace{\parindent}(ii) {\em There exist Petri curves of genus $10$ for which $B(2,12,4)\ne\emptyset$. Moreover $\beta(2,12,4)<0$.}

In section \ref{og}, we discuss curves of odd genus and obtain an interesting example for $g=5$ (Proposition \ref{prop7}). In section \ref{cs}, we show how our results can be interpreted in terms of coherent systems. Finally, in section \ref{ps}, we comment on the exciting developments that have taken place since this paper was completed.

Throughout the paper, $C$ is a smooth irreducible projective curve of genus $g\ge2$ defined over an algebraically closed field of characteristic $0$. The canonical bundle on $C$ is denoted by $K$. For any vector bundle $E$ on $C$, we write $d_E$ for the degree of $E$.

We thank the referee for a careful reading and some useful comments.
 
\section{Bundles of rank 2 with 4 sections}\label{rank2}
\begin{proposition}\label{prop1}
Let $C$ be a curve of genus $g\ge2$. Then
\begin{itemize}
\item[(i)] $\widetilde{B}(2,2d_1,4)\ne\emptyset$;
\item[(ii)] $B(2,d,4)\ne\emptyset$ if $d\ge2d_1+3$; moreover, if there exist two non-isomorphic line bundles $L_1$, $L_2$ of degree $d_1$ with $h^0(L_i)=2$, then $B(2,d,4)\ne\emptyset$ for $d\ge 2d_1+1$;
\item[(iii)] if $d<2d_1$, then $B(2,d,4)=\widetilde{B}(2,d,4)$;
\item[(iv)] if $e\le2d_1$ and $B(2,e,4)\ne\emptyset$, then $B(2,d,4)\ne\emptyset$ for $d\ge e$.
\end{itemize}
\end{proposition}
\begin{proof} (i) By definition, there exist line bundles on $C$ of degree $d_1$ with $h^0=2$. Let $L_1$, $L_2$ be two such line bundles. Then $L_1\oplus L_2$ is semistable of degree $2d_1$ with $h^0=4$.

(ii) Consider extensions
\begin{equation}\label{eqnelem}
0\rightarrow L_1\oplus L_2\rightarrow E\rightarrow\tau\rightarrow0,
\end{equation}
where $L_1$, $L_2$ are non-isomorphic line bundles of the same degree with $h^0(L_i)=2$ and $\tau$ is a torsion sheaf of length $t>0$. Such extensions are classified by pairs $(e_1,e_2)$, where $e_i\in \Ext^1(\tau,L_i)$. When $t=1$, it is easy to see that $E$ is a stable vector bundle provided $e_1$ and $e_2$ are both non-zero. For any $t>0$, it is in fact true that $E$ is a stable vector bundle for the general extension (\ref{eqnelem}) \cite[Th\'eor\`eme A.5]{M3}. This completes the proof when $L_1$, $L_2$ have degree $d_1$. When there exists only one line bundle $L$ of degree $d_1$ with $h^0(L)=2$, we take $L_1=L(p)$, $L_2=L(q)$ in \eqref{eqnelem}, where $p$, $q$ are distinct points of $C$.

(iii) Suppose that $d<2d_1$ and that $E$ is semistable of degree $d$. If $E$ is strictly semistable, then $d$ is even and there exists a line subbundle $L$ of $E$ of degree $\frac{d}2$. Now $\frac{d}2<d_1$, so $h^0(L)<2$ and $h^0(E/L)<2$. This contradicts the fact that $h^0(E)\ge4$, so $E$ cannot be strictly semistable.

(iv) It is sufficient to prove that $B(2,e+1,4)\ne\emptyset$, since we can then obtain the result by tensoring with effective line bundles. For this, let $E\in B(2,e,4)$ and consider extensions
\begin{equation}\label{eqntau}
0\rightarrow E\rightarrow F\rightarrow\tau\rightarrow0,
\end{equation}
where $F$ is a vector bundle and $\tau$ has length $1$. Every line subbundle $L$ of $E$ has $d_L<\frac{e}2\le d_1$, so $h^0(L)<2$. Hence every quotient line bundle $L'$ of $E$ has $h^0(L')\ge3$. It follows that $d_{L'}\ge d_2$, so $d_L\le e-d_2$. Thus every line subbundle of $F$ has degree $\le e-d_2+1<\frac{e+1}2=\mu(F)$ since $e\le2d_1\le2d_2-2$. So $F$ is stable.\end{proof}

\begin{remark}\begin{em}\label{rmk1}
For a Petri curve of genus $g\ge3$, there exist non-isomorphic  $L_1$, $L_2$ of degree $d_1$ with $h^0(L_i)=2$ and Proposition \ref{prop1}(ii) states that $B(2,d,4)$ is non-empty for $d\ge g+3$ when $g$ is even and for $d\ge g+4$ when $g$ is odd. This was previously proved by Teixidor for a general curve using degeneration methods \cite{T1}. This implies the result for odd $d$ for any curve on which $d_1$ takes its generic value $g+1-\left[\frac{g}2\right]$. On the other hand, for a hyperelliptic curve (and in particular any curve of genus $2$), there is only one line bundle $L$ of degree $d_1=2$ with $h^0=2$. The situation for $g=2$ is described in the following proposition.
\end{em}\end{remark}
\begin{proposition}\label{prop4}
Let $C$ be a curve of genus $2$. Then
\begin{itemize}
\item $\widetilde{B}(2,d,4)\ne\emptyset$ if and only if $d=4$ or $d\ge6$;
\item $B(2,d,4)\ne\emptyset$ if and only if $d\ge6$.
\end{itemize}
\end{proposition}
\pf Suppose first that $E$ is a semistable bundle of rank $2$ and degree $d$. If $d\le3$, then the Clifford theorem of \cite{X} gives $h^0(E)\le3$. For $d=5$, we have $h^0(E)=3$ by the Riemann-Roch theorem. On the other hand, if $d=4$ or $d=6$, there certainly exist semistable bundles $E$ with $h^0(E)=4$, namely $K\oplus K$ and $L_1\oplus L_2$, where $L_1$, $L_2$ are any line bundles of degree $3$. By Proposition \ref{prop1}(ii), $B(2,d,4)\ne\emptyset$ for $d\ge7$. Since $K\oplus K$ is the only semistable bundle of degree $4$ with $h^0=4$ (see \cite{Re}), 
it remains only  to show that there exists a stable bundle $E$ of degree $6$ with $h^0(E)\ge4$. For this we consider extensions
\begin{equation}\label{eqng2}
0\rightarrow L\rightarrow E\rightarrow M\rightarrow0,
\end{equation}
where $\deg L=2$ and $\deg M=4$. Suppose that $L\not\cong K$; then $h^0(L)=1$ and $h^1(L)=0$. On the other hand, $h^0(M)=3$, so it follows from the cohomology sequence of (\ref{eqng2}) that $h^0(E)=4$. Moreover, it is a standard fact (see, for example, \cite{LN}) that, for the general extension (\ref{eqng2}), $E$ is stable. \qed

We now return to the case of an arbitrary curve $C$ of genus $g\ge2$.
\begin{proposition}\label{prop2}
Suppose that $E\in B(2,d,4)$ with $d\le2 d_1$. Let $E'$ be the image of the evaluation morphism $H^0(E)\otimes {\cO} \rightarrow E$. Then
\begin{itemize}
\item[(i)] $h^0(L)\le1$ for every line subbundle $L$ of $E$;
\item[(ii)] $h^0(\det E')\ge5$;
\item[(iii)] $d\ge d_4$.
\end{itemize}
\end{proposition}
\pf (i)
Since $E$ is stable of degree $d\le 2d_1$, every line subbundle $L$ of $E$ has degree $d_L< d_1$ and so $h^0(L)\le1$.

(ii) [This follows directly from (i) and \cite[Lemma 3.9]{PR}; for the convenience of the reader and for future reference, we include a proof.] It follows from (i) that $E'$ is of rank $2$ and $h^0(E'^*)=0$
(otherwise $\cO$ would be a direct summand of $E'$ and the other factor would be a line bundle $L$ with $h^0(L)\ge3$). If $E'=E$, then $E'$ is stable by hypothesis. If $E'\not=E$, any proper quotient of $E'$ is a non-trivial generated line bundle and therefore has degree $\ge d_1>\mu(E')$. So again $E'$ is stable.

Now consider the exact sequence
\begin{equation}\label{eqn:1}
0\rightarrow D_V(E' )^* \rightarrow V\otimes  \cO \rightarrow E'\rightarrow0,
\end{equation}
where $V$ is a $4$-dimensional subspace of $H^0(E')$ which generates $E'$.
Applying $\bigwedge^2$, we obtain the two sequences
\begin{equation}\label{eqn:2}
0\rightarrow F\rightarrow \bigwedge^2 V\otimes\cO\rightarrow \det E' \rightarrow 0
\end{equation}
and
\begin{equation}\label{eqn:3}
0\rightarrow \bigwedge^2 D_V(E')^*\rightarrow F\rightarrow D_V(E')^*\otimes E' \rightarrow 0.
\end{equation}
Now, from (\ref{eqn:1}), $D_V(E')$ has rank $2$, is generated and satisfies $h^0(D_V(E'))\ge4$ and $h^0(D_V(E')^*)=0$. It follows that any proper quotient of $D_V(E')$ is a non-trivial generated line bundle and therefore has degree $\ge d_1\ge\mu(D_V(E'))$. So $D_V(E')$ is at least semistable. Since $\mu(D_V(E'))=\mu(E')$ and $E'$ is stable, this implies that $h^0(D_V(E')^*\otimes E')\le1$. Hence, by (\ref{eqn:3}), $h^0(F)\le1$ and, by (\ref{eqn:2}), $h^0(\det E')\ge5$.

(iii) Since $\det E'$ is a subsheaf of $\det E$, (iii) follows from (ii) and the definition of $d_4$.\qed

\begin{remark}\label{r1}\begin{em}
Under the same hypotheses, suppose that $h^0(F)>0$, where $F$ is defined by (\ref{eqn:2}). Then $h^0(D_V(E')^*\otimes E')>0$ by (\ref{eqn:3}). From the proof of Proposition \ref{prop2}(ii), we see that $E'$ is stable and $D_V(E')$ is semistable, and these two bundles have the same slope. So $D_V(E')\cong E'$ and $h^0(F)=1$.
\end{em}\end{remark}

We have the following corollary, which is also an immediate consequence of \cite[Theorem 5.2]{ln}.

\begin{corollary}\label{cor1}
Let $E$ be a semistable bundle of rank $2$ with $h^0(E)=4$. Then $\gamma(E)\ge\min\{\Cl,\frac{d_4}2-2\}$.
\end{corollary}
\begin{proof}
From the definition of $\gamma(E)$, we see that it is sufficient to show that $d_E\ge\min\{2d_1,d_4\}$. This follows at once from Propositions \ref{prop2}(iii) and \ref{prop1}(iii).
\end{proof}

\begin{remark}\begin{em}\label{r8}
Propositions \ref{prop1} and \ref{prop2} lead to the following trichotomy:
\begin {itemize}
\item[(a)] $d_4>2d_1$. If also there exist two non-isomorphic line bundles of degree $d_1$ with $h^0=2$, then $\widetilde{B}(2,d,4)\ne\emptyset$ if and only if $d\ge2d_1$ and $B(2,d,4)\ne\emptyset$ if and only if $d\ge2d_1+1$. If there is only one line bundle of degree $d_1$ with $h^0=2$, the same holds except that $B(2,2d_1+1,4)$ and $B(2,2d_1+2,4)$  could possibly be empty (see Proposition \ref{prop4} for a case in which $B(2,2d_1+1,4)=\emptyset$).
\item[(b)] $d_4=2d_1$. The same holds as in case (a) except that it is now possible that $B (2,2d_1,4)\ne\emptyset$; when this happens, $B(2,d,4)\ne\emptyset$ if and only if $d\ge2d_1$ (see Proposition \ref{prop1}(iv)).  
\item[(c)] $d_4<2d_1$. Let $d_{\min}$ be the minimum value of $d$ for which $\widetilde{B}(2,d,4)\ne\emptyset$. Then, by Propositions \ref{prop1} and \ref{prop2}(iii), $d_4\le d_{\min}\le2d_1$. If $d_{\min}<2d_1$, then $B(2,d,4)\ne\emptyset$ if and only if $d\ge d_{min}$. If $d_{min}=2d_1$, the situation  is the same as in case (b).
\end{itemize}\end{em}\end{remark}

\begin{example}\begin{em}\label{ex1}
If $C$ is a Petri curve of genus $g$, we have $d_r=g+r-\left[\frac{g}{r+1}\right]$. It follows that, if $g$ is even, then $d_4>2d_1$ for $g\le8$. For $g=4,6,8$, case (a) applies and $B(2,d,4)\ne\emptyset$ if and only if $d\ge2d_1+1$. If $g=10,12,14$, then case (b) applies and it may happen that $B(2,2d_1,4)\ne\emptyset$. The case $g=10$ will be discussed fully in section \ref{g=10}.
\end{em}\end{example}
\begin{example}\begin{em}\label{ex2}
If $C$ is a smooth plane curve of degree $\delta\ge5$, then $d_1=\delta-1$ and $d_4=2\delta-1$, so we are in case (a). Moreover there exist infinitely many line bundles of degree $d_1$ with $h^0=2$, so $\widetilde{B}(2,d,4)\ne\emptyset$ if and only if $d\ge2\delta-2$ and $B(2,d,4)\ne\emptyset$ if and only if $d\ge2\delta-1$.
\end{em}\end{example}
\begin{example}\begin{em}\label{ex3}
If $C$ is a curve of Clifford dimension $\ge4$ (i.e. none of $d_1$, $d_2$, $d_3$ computes $\Cl$), then we are in case (c). In fact, $d_r$ computes $\Cl$ for some $r\ge4$. By \cite[Corollary 3.5]{elms}, $d_r\ge 4r-3>3r$. So
$$d_4\le d_r-r+4<2d_r-4r+4=2\Cl+4<2d_1.$$
If $C$ has Clifford dimension $3$, then $C$ is an intersection of $2$ cubics in ${\mathbb P}^3$ (see \cite{mar}) and has genus $10$ with $\Cl=3$. It follows that $d_1=6$ and $d_4=12$, so we are in case (b).
\end{em}\end{example}

\section{Bundles of degree $d_4$}\label{d4}
In this section, we consider the case where $d_4\le2d_1$ and $d_E=d_4$ and state and prove our first main theorem.
\begin{lemma}\label{lem1}
Suppose $d_4\le2d_1$ and $E\in B(2,d_4,4)$, then
\begin{itemize}
\item $h^0(E)=4$;
\item $h^0(\det E)=5$ and $\det E$ is generated;
\item $E$ is generated;
\item the canonical homomorphism $\bigwedge^2H^0(E)\to H^0(\det E)$ is surjective;
\item $D_{H^0(E)}(E)\cong E$.
\end{itemize}
\end{lemma}
\pf Suppose $E\in B(2,d_4,4)$ with $h^0(E)=2+s$. By Proposition  \ref{prop2}(i), we know that $h^0(L)\le1$ for every line subbundle $L$ of $E$. It follows from \cite[Lemma 3.9]{PR} that $h^0(\det E)\ge2s+1$ and therefore $d_E=d_{\det E}\ge d_{2s}$. This is a contradiction if $s>2$; so $h^0(E)=4$. Proposition \ref{prop2}(ii) implies that $h^0(\det E)\ge5$ and that $d_{\det E'}\ge d_4$. Since $d_E=d_4$, this shows that $h^0(\det E)=5$ and $\det E$ is generated; moreover $E'=E$ and $E$ is generated. It follows from (\ref{eqn:2})  that $h^0(F)>0$. The remaining parts of the lemma now follow from Remark \ref{r1} and (\ref{eqn:2}).\qed
\begin{theorem}\label{main}
Suppose $d_4\le2d_1$. Then
\begin{itemize}
\item[(i)] if $B(2,d_4,4)\ne\emptyset$, there exists a line bundle $M\in B(1,d_4,5)$ such that $\phi_M(C)$ is contained in a quadric;
\item[(ii)] if $d_4<2d_1$ and there exists $M\in B(1,d_4,5)$ such that $\phi_M(C)$ is contained in a quadric, then $B(2,d_4,4)\ne\emptyset$. 
\end{itemize}
\end{theorem}

\pf (i) Let $E\in B(2,d_4,4)$. By Lemma \ref{lem1}, $E$ is generated with $h^0(E)=4$, so we have a morphism 
$$\phi_E: C\longrightarrow \mbox{Gr}(2,4)\hookrightarrow {\mathbb P}(\bigwedge^2H^0(E)^*).$$
The line bundle $\det E$ is also generated and $h^0(\det E)=5$, so we have a morphism $$\phi_{\det E}:C\longrightarrow {\mathbb P}(H^0(\det E)^*).$$
Moreover the canonical homomorphism $\bigwedge^2H^0(E)\to H^0(\det E)$ is surjective, so there is a canonical embedding of ${\mathbb P}(H^0(\det E)^*)$ in ${\mathbb P}(\bigwedge^2H^0(E)^*)$ as a hyperplane, giving a commutative diagram
$$\begin{array}{ccc}
C&\stackrel{\phi_E}{\longrightarrow}&{\mathbb P}(\bigwedge^2H^0(E)^*)={\mathbb P}^5\\
\parallel&&\cup\\
C&\stackrel{\phi_{\det E}}{\longrightarrow}&{\mathbb P}(H^0(\det E)^*)={\mathbb P}^4.
\end{array}$$
It follows that $\phi_{\det E}(C)$ is contained in the quadric $\mbox{Gr}(2,4)\cap{\mathbb P}^4$. 

(ii) Suppose that $M$ satisfies the hypothesis of (ii) and let $q$ be a quadric containing $\phi_M(C)$ defined by an element $[q]$ of the kernel of the canonical map
\begin{equation}\label{eqnpsi}
\psi:S^2H^0(M)\longrightarrow H^0(M^2).
\end{equation}
We can choose a linear subspace $W$ of $H^0(M)$ of dimension $3$ such that $W$ generates $M$ and the kernel $N$ of the linear map $W\otimes H^0(M)\to H^0(M^2)$ has dimension $\ge4$. For this, first choose a line $\ell$ on $q$ which does not meet $\phi_M(C)$ (a simple dimensional calculation shows that such lines exist). Now let $W$ be the 3-dimensional subspace of $H^0(M)$ which annihilates the 2-dimensional subspace of $H^0(M)^*$ defined by $\ell$. The fact that $\ell$ does not meet $\phi_M(C)$ implies that $W$ generates $M$, while the fact that $\ell$ lies on $q$ implies that the image of $W\otimes H^0(M)$ in $S^2H^0(M)$ contains $[q]$. It follows that the image of $N$ in $S^2H^0(M)$ contains $[q]$. Since the kernel of the map $N\to S^2H^0(M)$ is $\bigwedge^2W$, it follows at once that $\dim N\ge4$.

Now define $E^*$ to be the kernel of the evaluation map $W\otimes\cO\to M$, so that we have an exact sequence
\begin{equation}\label{eqn9}
0\longrightarrow E^*\longrightarrow W\otimes\cO\longrightarrow M\longrightarrow0
\end{equation}
and $\det E\cong M$. Tensoring (\ref{eqn9}) by $M$ and recalling that $E^*\otimes M\cong E$, we obtain
\begin{equation}\label{eqn10}
0\longrightarrow E\longrightarrow W\otimes M\longrightarrow M^2\longrightarrow0.
\end{equation}
So $H^0(E)\cong N$ and $h^0(E)\ge4$. By (\ref{eqn9}), $h^0(E^*)=0$ and $E$ is generated, so any quotient line bundle $L$ of $E$  has $h^0(L)\ge2$. So $d_L\ge d_1$ and $E$ is stable. \qed

\begin{remark}\label{r2}\begin{em}
(i)  If $d_4=2d_1$, the argument of (ii) is still valid except that there is now the possibility that $E$ is strictly semistable. In this case, we have an exact sequence
$$0\longrightarrow L\longrightarrow E\longrightarrow L'\longrightarrow0,$$
with $L$, $L'$ of degree $d_1$ with $h^0(L)=h^0(L')=2$; moreover $L\otimes L'\cong M$. If $L\cong L'$ and $C$ is Petri, this implies that $h^1(M)=0$ and hence $h^0(M)\le4$ since $2d_1\le g+3$; this is a contradiction. 
 For a general curve of even genus $g\ge4$, Voisin has shown \cite[Proposition 4.1(i)]{Voi} that, if $L\not\cong L'$, then $h^1(L\otimes L')=1$; since $d_{L\otimes L'}=2d_1=g+2$, it follows that $h^0(L\otimes L')=4$, again a contradiction. The condition $d_4=2d_1$ holds for general curves of genera $10$, $12$ and $14$, so in these cases we can say that $B(2,g+2,4)\ne\emptyset$ if and only if there exists $M\in B(1,g+2,5)$ such that $\phi_M(C)$ is contained in a quadric.

(ii) Suppose $d_4=2d_1$, $M\in B(1,d_4,5)$ and $\phi_M(C)$ is contained in a quadric $q$. Then $q$ cannot have rank $\le2$ since $\phi_M(C)$ is non-degenerate. If $q$ has rank $3$ or $4$, then $M\cong L\otimes L'$ with $h^0(L)=h^0(L')=2$ (the line bundles $L$ and $L'$ correspond to the two systems of planes lying on $q$ when the rank is $4$; when the rank is $3$, we have only one system of planes and $L\cong L'$). It follows that, if $M\not\cong L\otimes L'$, then $\phi_M(C)$ cannot be contained in a quadric of rank $\le4$. Since every pencil of quadrics in $\PP^4$ contains a member of rank $\le4$, this means that, in this case, any quadric containing $\phi_M(C)$ is unique and has rank $5$. When $C$ is Petri, $\phi_M(C)$ can never be contained in a quadric of rank $3$, since as we have seen, $M\cong L^2$ with $h^0(L)\ge2$ implies $h^1(M)=0$ and hence $h^0(M)\le4$.

(iii) It is of interest to note that the kernel of the map $\psi$ of \eqref{eqnpsi} can be identified with the Koszul cohomology group $K_{1,1}(C,M)$.
\end{em}\end{remark}

\begin{remark}\begin{em}\label{r5}
Suppose $d_4<2d_1$, let $E\in B(2,d_4,4)$ and write $M=\det E$. For any subspace $V$ of $H^0(E)$ of dimension $3$ which generates $E$, we have an exact sequence
$$0\longrightarrow M^*\longrightarrow V\otimes{\mathcal O}\longrightarrow E\longrightarrow 0.$$
Dualising and using the fact that $h^0(E^*)=0$, we see that  $V^*$ is a subspace of $H^0(M)$. Now tensoring by $M$ and writing $W=V^*$ yields a sequence \eqref{eqn10}. It follows that $H^0(E)$ can be identified with the kernel $N$ of the map $W\otimes H^0(M)\to H^0(M^2)$. The image of $N$ in $S^2H^0(M)$ has dimension $1$ and therefore coincides with the kernel of the map $\psi$ of \eqref{eqnpsi}, which defines the unique quadric $q$ on which $\phi_M(C)$ lies (compare Remark \ref{r2}(ii)). As in the proof of Theorem \ref{main}, we see that $W$ corresponds to a line $\ell$ on $q$ not meeting $\phi_M(C)$ and also that we can recover $(E,V)$ from $(M,\ell)$. Thus we obtain a bijective correspondence between the sets
$$T_1:=\{(E,V)|E\in B(2,d_4,4), V\subset H^0(E), \dim V=3, V\mbox{ generates }E\}$$
and
$$T_2:=\{(M,\ell)|M\in B(1,d_4,5)|, \phi_M(C)\subset q, \ell \mbox{ a line on }q\mbox{ not meeting }\phi_M(C)\}.$$
The sets $T_1$, $T_2$ have natural structures of quasi-projective variety and the bijective correspondence then becomes an isomorphism $f:T_1\to T_2$. The subvarieties of $T_1$, $T_2$ corresponding to fixed $E$ and $M$ respectively are (if non-empty) both 3-dimensional irreducible quasi-projective varieties. Note further that $f(E,V)=(\det E,\ell)$ for some $\ell$. It follows that $f$ induces an embedding of $B(2,d_4,4)$ in $B(1,d_4,5)$.

If $d_4=2d_1$, the same argument works, but we must now exclude from $T_1$ those $(E,V)$ for which $\det E\cong L\otimes L'$ and from $T_2$ those $(M,\ell)$ for which $M\cong L\otimes L'$, where in both cases $h^0(L)=h^0(L')=2$.
\end{em}\end{remark}

\begin{remark}\label{r3}\begin{em}
If $d_1$ computes $\Cl$ and $d_4<2d_1$, then any $E\in B(2,d_4,4)$ has $\gamma(E)<\Cl$. Moreover, by Corollary \ref{cor1}, $\gamma(E)=\Clt$. So this is a possible source of examples for which $\Clt<\Cl$.
\end{em}\end{remark}

\begin{remark}\begin{em}\label{r4}
For a Petri curve $C$, the expected dimension of $B(2,d_4,4)$ is
$$\beta(2,d_4,4)=4g-3-4(4-d_4+2g-2).$$
Substituting $d_4=g+4-\left[\frac{g}5\right]$, we obtain
$$\beta(2,d_4,4)=5-4\left[\frac{g}5\right].$$
For $g\ge10$, this is negative and also $d_4\le2d_1$. If we fix $M\in B(1,d_4,5)$, we have $h^1(M)\ge2$ since $d_4\le g+2$. Osserman's result \cite[Theorem 1.3]{Oss} gives a lower bound 
$$\beta(2,d_4,4)-g+12=17-g-4\left[\frac{g}5\right]$$
for the dimension of the variety $\{E\in B(2,d_4,4)|\det E\cong M\}$,
which is still negative. So, if $B(2,d_4,4)\ne\emptyset$, this would show that Osserman's bound is not sharp even in the case where $h^1(M)=2$. 
\end{em}\end{remark}

In many cases, we do however have an upper bound for $\dim B(2,d_4,4)$.

\begin{proposition}\label{newprop}
Suppose $C$ is a Petri curve and either $d_4<2d_1$ or $d_4=2d_1$ and no $M\in B(1,d_4,5)$ has the form $M\cong L\otimes L'$ with $h^0(L)=h^0(L')=2$. Then
$$\dim B(2,d_4,4)\le g-5\left[\frac{g}5\right].$$

\end{proposition}
\begin{proof}
By Remark \ref{r5}, we have
$$\dim B(2,d_4,4)\le\dim B(1,d_4,5)=g-5\left[\frac{g}5\right].$$
\end{proof}

This holds in particular for any Petri curve of genus $11$, $13$ or $\ge15$ since then $d_4<2d_1$. It holds also for the general curve of genus $10$, $12$ or $14$ by Remark \ref{r2}.

\section{The case $g=10$}\label{g=10}
In this section, we prove the original main theorem of this paper as an application of Theorem \ref{main}.
\newpage
\begin{theorem}\label{main2}
\begin{itemize}
\item[]
\item[(i)] Let $C$ be a general curve of genus $10$. Then $$B(2,d,4)\ne\emptyset\Longleftrightarrow d\ge13\Longleftrightarrow \beta(2,d,4)\ge0.$$
\item[(ii)] There exist Petri curves of genus $10$ for which $B(2,12,4)\ne\emptyset$. Moreover $\beta(2,12,4)<0$.
\end{itemize}
\end{theorem}

\pf (i) For a general curve of genus $10$, we have $d_1=6$ and $d_4=12$; moreover, there exist non-isomorphic line bundles of degree $6$ with $h^0=2$ (in fact, by Castelnuovo's formula \cite[Ch V, formula (1.2)]{ACGH}, there are precisely $42$ isomorphism classes of such bundles). It follows from Propositions \ref{prop1} and \ref{prop2} that $B(2,d,4)=\emptyset$ for $d<12$ and $B(2,d,4)\ne\emptyset$ for $d>12$. It remains to consider $B(2,12,4)$.

Suppose that $B(2,12,4)\ne\emptyset$. Then, by Theorem \ref{main}(i), there exists a line bundle $M\in B(1,12,5)$ such that $\phi_M(C)$ is contained in a quadric, in other words the map $\psi$ of \eqref{eqnpsi} is not injective.  Since $S^2H^0(M)$ and $H^0(M^2)$  both have dimension $15$ (the first since $h^0(M)=5$, the second by Riemann-Roch), this means that $\psi$ is not surjective. However, for a general curve $C$, $\psi$ is always surjective \cite[Proposition 4.2]{Voi}. Hence $B(2,12,4)=\emptyset$.

Finally $\beta(2,d,4)=4d-51\ge0\Longleftrightarrow d\ge13$.

(ii) For any genus $g$, there exists a Petri curve $C$ which lies on a K3 surface $S$ and whose class generates $\mbox{Pic}\,S$ \cite{Laz}. For $g$ even, $g\ge10$, it follows from the proofs of \cite[Propositions 4.1, 4.12]{Voi} that the hypotheses of \cite[3.1]{Voi} are satisfied. Now let $M\in B(1,12,5)$. Then, by \cite[Theorem 0.3, Proposition 3.2]{Voi} (our $M$ becomes $K_C-L$ in Voisin's notation) (see also \cite{W}), the map $\psi$
is not surjective. When $g=10$, this implies as above that $\psi$ is not injective and so $\phi_M(C)$ lies on a quadric. It follows from Theorem \ref{main}(ii) and Remark \ref{r2} that $B(2,12,4)\ne\emptyset$. Finally, $\beta(2,12,4)=-3<0$. 
\qed

\begin{proposition}\label{prop6}
Let $C$ be a general curve of genus $10$. Then there exists a stable bundle $E$ computing $\Clt$ if and only if there exist line bundles $M$, $M'$ with 
\begin{equation}\label{eqn12}
d_M=d_{M'}=12,\  h^0(M)=h^0(M')=5
\end{equation} 
such that
\begin{equation}\label{eqn11}
H^0(M)\otimes H^0(M')\to H^0(M\otimes M')
\end{equation}
is not surjective.
\end{proposition}
\begin{proof}
We know already that, for a general curve of genus $10$, $\Clt=\Cl=4$ (see \cite[Proposition 3.8]{ln}). Theorem \ref{main2} shows that a stable bundle $E$ of rank $2$ with $h^0(E)=4$ has $\gamma(E)\ge\frac92$. Thus, if $E$ is a stable bundle computing $\Clt$, we must have $h^0(E)=2+s$ with $s\ge3$; moreover, by definition of $\Clt$, $d_E\le2g-2=18$. If $E$ has no line subbundle with $h^0(L)\ge2$, then, by \cite[Lemma 3.9]{PR}, 
$$d_E\ge d_{2s}\ge d_6+2s-6=9+2s$$
and so $\gamma(E)\ge\frac92$ again. On the other hand, if $E$ has a line subbundle $L$ with $h^0(L)\ge3$, then, by stability, $d_E>2d_L\ge2d_2=18$,
a contradiction. It follows that $E$ can be expressed as a non-trivial extension
\begin{equation}\label{eqn13}
0\to L\to E\to M\to0,
\end{equation}
where $h^0(L)=2$ and hence $d_L\ge d_1=6$. We therefore require $d_{s-1}\le d_M\le12$. Since $d_4=12$, this implies $s\le5$. The values $s=3$ and $s=4$ give $d_M\ge9$ and $d_M\ge11$ respectively and a simple calculation gives $\gamma(E)\ge\frac92$ once more. The only remaining possibility is $s=5$, in which case we could have $\gamma(E)=4$ if all the above inequalities are equalities. Writing $M'=K\otimes L^*$, this is equivalent to \eqref{eqn12}. Moreover $h^0(E)=7$, so all the sections of $M$ lift to $E$; it follows that the class of the extension \eqref{eqn13} belongs to the kernel of the canonical map
$$H^1(M^*\otimes L)\to \Hom(H^0(M),H^1(L)).$$
Thus this map is not injective. Dualising, we see that \eqref{eqn11} is not surjective.

Conversely, suppose $M$ and $M'$ exist satisfying \eqref{eqn10} such that  \eqref{eqn11} is not surjective. Then, if we write $L=K\otimes M'^*$, there exists a non-trivial extension \eqref{eqn13} with $h^0(E)=7$. Then $\gamma(E)=4$ and it is easy to check that $E$ is stable. 
\end{proof}

It is an interesting question as to whether line bundles $M$, $M'$ as in Proposition \ref{prop6} can exist. Note that,  when $M=M'$, \eqref{eqn11} is equivalent to asserting that the map $\psi$ of \eqref{eqnpsi} is not surjective. This cannot happen on a general curve of genus $10$ by \cite[Proposition 4.2]{Voi}.

\section{Curves of odd genus}\label{og}
Let $C$ be a Petri curve of odd genus $g\ge3$. Note that, for such a curve, $d_1=\frac{g+3}2$, so
$$\beta(1,d_1,2)=g-2\left(2-\frac{g+3}2+g-1\right)=1.$$
There are therefore infinitely many non-isomorphic line bundles $L$ of degree $d_1$ with $h^0(L)=2$.

For $g=3$, we have $d_1=3$, $d_4=7$, so we are in case (a) of Remark \ref{r8} and
\begin{itemize}
\item $\widetilde{B}(2,d,4)\ne\emptyset$ if and only if $d\ge6$.
\item $B(2,d,4)\ne\emptyset$ if and only if $d\ge7$.
\end{itemize}
For $g=5,7$ and $9$, we are in case (b) with $d_4=g+3=2d_1$. 

\begin{proposition}\label{prop7} Let $C$ be a Petri curve of genus $g=5,7$ or $9$. Then
\begin{itemize}
\item[(i)] $\widetilde{B}(2,d,4)\ne\emptyset$ if and only if $d\ge g+3$;
\item[(ii)] $B(2,d,4)\ne\emptyset$ if $d\ge g+4$;
\item[(iii)] if $g=5$, $B(2,d,4)\ne\emptyset$ if and only if $d\ge 8$; moreover, if $E\in B(2,8,4)$, then $\det E\cong K$ and $\dim B(2,8,4)=2$.
\end{itemize}
\end{proposition}
\begin{proof} (i) and (ii). Since $d_4=2d_1$ and there are infinitely many non-isomorphic line bundles $L$ of degree $d_1$ with $h^0(L)=2$, it follows from Propositions \ref{prop1} and \ref{prop2} that (i) and (ii) hold.

(iii) Suppose now $g=5$ and $E\in B(2,8,4)$; then by Lemma \ref{lem1}, $h^0(\det E)=5$. It follows at once that $\det E\cong K$. The other parts of (iii) now follow from \cite{BF}.
\end{proof}

\begin{remark}\label{r9}\begin{em}
For $g=5$, the strictly semistable locus $\widetilde{B}(2,8,4)\setminus B(2,8,4)$ is isomorphic to $S^2B(1,4,2)$ and therefore has dimension $2$. Thus $\widetilde{B}(2,8,4)$ has two components, each of dimension $2$, which intersect in a subvariety of dimension $1$ (the points of this subvariety are given by bundles of the form $L\oplus(K\otimes L^*)$ with $L\in B(1,4,2)$).   On the other hand, $\beta(2,8,4)=1$.
\end{em}\end{remark}

\begin{remark}\label{r10}\begin{em}
For $g=7$ and $g=9$, the argument from \cite{BF} does not apply. For any $M\in B(1,g+3,5)$, we have $h^0(M^2)=g+7$. For $g=7$, it follows that the map $\psi$ of \eqref{eqnpsi} cannot be injective, so $\phi_M(C)$ is contained in a quadric. However, since $\beta(1,10,5)=2$, it is possible that, for all $M\in B(1,10,5)$, $M\cong L\otimes L'$ for some $L, L'\in B(1,5,2)$ and the method of Theorem \ref{main} may yield only strictly semistable bundles. For $g=9$, we have $\beta(1,12,5)=4$, so there exist $M\in B(1,12,5)$ which are not of the form $L\otimes L'$. However, we now have $h^0(M^2)=16$, so $\phi_M(C)$ may not be contained in a quadric and again the argument of Theorem \ref{main} and Remark \ref{r2} may not yield any bundles in $B(2,12,4)$.  
\end{em}\end{remark}

\begin{remark}\label{r11}\begin{em}
If $g\ge11$, we are in case (c) and $d_4<2d_1$. It follows that any $E\in B(2,d_4,4)$ has $\gamma(E)<\Cl$. Moreover, by \cite[Theorem 5.2]{ln}, $\gamma(E)=\Clt$. In particular, in genus $11$, we have $d_1=7$ and $d_4=13$. According to Theorem \ref{main}, $B(2,13,4)\ne\emptyset$ if and only if there is a non-degenerate morphism $C\to{\mathbb P}^4$ of degree $13$ whose image is contained in a quadric.  It is an interesting question to determine whether this is true for all Petri curves of genus $11$, for some but not all Petri curves or never (for further developments, see section \ref{ps}).
\end{em}\end{remark}

\section{coherent systems}\label{cs}
We recall that a {\em coherent system} of type $(n,d,k)$ on $C$ is a pair $(E,V)$, where $E$ is a vector bundle of rank $n$ and degree $d$ and $V$ is a $k$-dimensional subspace of $H^0(E)$. For any real number $\alpha>0$, the coherent system is {\em  $\alpha$-stable} if, for every proper subsystem $(F,W)$, with $F$ of rank $r$,
$$\frac{d_F}{r}+\alpha\frac{\dim W}{r}<\frac{d}n+\alpha\frac{k}n.$$
There exists a coarse moduli space $G(\alpha;n,d,k)$ for the $\alpha$-stable coherent systems of type $(n,d,k)$ (see \cite{BGMN} for a discussion of the general theory of coherent systems on curves).

A possible way of constructing bundles in $B(2,d,4)$ is to construct coherent systems of type $(2,d,4)$ and then use the methods of \cite{BGMN} to show that the underlying bundles are stable. It turns out that this doesn't help with the problem considered earlier, but one still obtains interesting results by interpreting the problem in terms of coherent systems.

We start by recalling a definition from \cite{BP}:
$$U(n,d,k):=\{(E,V)|(E,V)\in G(\alpha;n,d,k) \mbox{ for all }\alpha>0\mbox{ and }E\mbox{ is stable}\}.$$
\begin{proposition}\label{prop5}
Suppose $d\le2d_1+1$ and $(E,V)$ is a coherent system of type $(2,d,4)$. The following are equivalent:
\begin{itemize}
\item[(i)] $(E,V)\in G(\alpha;2,d,4)$ for some $\alpha>0$;
\item[(ii)] $E$ is stable;
\item[(iii)] $(E,V)\in U(2,d,4)$.
\end{itemize}
\end{proposition}
\begin{proof} (i)$\Rightarrow$(ii). If $(E,V)$ is $\alpha$-stable and $L$ is a subbundle with $d_L\ge\frac{d}2$, then $d_{E/L}\le d_1$, so $h^0(E/L)\le2$ and hence $\dim(H^0(L)\cap V)\ge2$. This contradicts the $\alpha$-stability of $(E,V)$, proving that $E$ is stable. 

(ii)$\Rightarrow$(iii). Every line subbundle $L$ of $E$ has degree $\le d_1$ and hence $h^0(L)\le2$. It follows that $(E,V)$ is $\alpha$-stable for all $\alpha>0$.

(iii)$\Rightarrow$(i) is obvious.
\end{proof}
\begin{remark}\begin{em}\label{r7}
One can show easily that, under the hypotheses of Proposition \ref{prop5}, there are no actual critical values in the sense of \cite[Definition 2.4]{BGMN}.
\end{em}\end{remark}

\begin{remark}\label{rlet}\begin{em}
On a general curve of genus $g\ge3$, or more generally whenever $d_1$ computes $\Cl$, Proposition \ref{prop5} is proved for $d<2d_1$ in \cite[Theorem 1.1]{BO}.
\end{em}\end{remark}

To handle coherent systems whose degree is greater than $2d_1+1$, the following lemma will be useful.

\begin{lemma}\label{lem+1}
Let $(L_1,V_1)$, $(L_2,V_2)$ be coherent systems of type $(1,d_L,2)$ and
$$0\longrightarrow L_1\oplus L_2\longrightarrow E\longrightarrow\tau\longrightarrow0$$
an extension with $\tau$ a torsion sheaf and $E$ a stable bundle. Then 
$$(E,V_1\oplus V_2)\in U(2,d_E,4).$$
\end{lemma} 
\begin{proof} Suppose $(F,W)$ is a coherent subsystem of $(E,V_1\oplus V_2)$ with $F$ of rank $1$. Then $W\subset V_1\oplus V_2$, so $(F\cap(L_1\oplus L_2), W)$ is a coherent subsystem of $(L_1\oplus L_2,V_1\oplus V_2)$; it follows that $\dim W\le2$. On the other hand, by stability of $E$, $d_F<\frac{d_E}2$. So $(E,V_1\oplus V_2)$ is $\alpha$-stable for all $\alpha>0$.
\end{proof}

\begin{remark}\label{r15}\begin{em}
The stable bundles $E$ constructed in the proof of Proposition \ref{prop1}(ii) all satisfy the conditions of Lemma \ref{lem+1} and therefore give rise to coherent systems $(E,V)\in U(2,d_E,4)$.
\end{em}\end{remark}

We can now interpret our results in terms of coherent systems, beginning with the case $g=2$.

 \begin{theorem}\label{cs1}
 Let $C$ be a curve of genus $2$. Then
 \begin{itemize}
 \item[(i)] if $d\le5$, then $G(\alpha;2,d,4)=\emptyset$ for all $\alpha>0$;
 \item[(ii)] if $d\ge6$, then $U(2,d,4)\ne\emptyset$.
 \end{itemize}
\end{theorem}
\begin{proof} (i) follows at once from Propositions \ref{prop4} and \ref{prop5}.

(ii) For $d\ge7$, this follows from Proposition \ref{prop1}(ii) and Remark \ref{r15}. For $d=6$,  $(E,V)$ can be constructed from the sequence \eqref{eqng2} with $E$ stable. Any line subbundle $L'$ of $E$ has degree $d_{L'}\le2$ and therefore $h^0(L')\le2$; hence $(E,V)\in U(2,6,4)$.
\end{proof}

\begin{remark}\label{r16}\begin{em}
For $d\le3$ and $d\ge6$, this is proved in \cite{BP}; for $d=4$ and $d=5$, the result in \cite{BP} is slightly weaker than ours.
\end{em}\end{remark}

\begin{theorem}\label{cs2}
For any curve $C$,
\begin{itemize}
\item[(i)] $U(2,d,4)\ne\emptyset$ if $d\ge2d_1+3$;
\item[(ii)] if there exist two non-isomorphic line bundles $L_1$, $L_2$ of degree $d_1$ with $h^0(L_i)=2$, then $U(2,d,4)\ne\emptyset$ for $d\ge 2d_1+1$;
\item[(iii)] if $e\le2d_1$ and $B(2,e,4)\ne\emptyset$, then $U(2,d,4)\ne\emptyset$ for $d\ge e$;
\item[(iv)] if $d<\min\{2d_1,d_4\}$, then $G(\alpha;2,d,4)=\emptyset$ for all $\alpha>0$.
\end{itemize}
\end{theorem}
\begin{proof} (i) and (ii) follow from Proposition \ref{prop1}(ii) and Remark \ref{r15}.

(iii) We need only prove that $U(2,e,4)$ and $U(2,e+1,4)$ are non-empty, since the results for $d\ge e+2$ can then be obtained by tensoring by effective line bundles. For these cases, non-emptiness follows from Propositions \ref{prop1}(iv) and \ref{prop5}.

(iv) follows from Propositions \ref{prop2}(iii) and \ref{prop5}.
\end{proof}

\begin{remark}\label{r17}\begin{em}
Theorem \ref{cs2}(ii) was proved for a general curve by Teixidor using degeneration methods \cite{T2}.
\end{em}\end{remark}

\begin{theorem}\label{cs3}
Suppose $d_4\le2d_1$. Then
\begin{itemize}
\item[(i)] if $G(\alpha;2,d_4,4)\ne\emptyset$ for some $\alpha>0$, there exists a line bundle $M\in B(1,d_4,5)$ such that $\phi_M(C)$ is contained in a quadric;
\item[(ii)] if $d_4<2d_1$ and there exists $M\in B(1,d_4,5)$ such that $\phi_M(C)$ is contained in a quadric, then $U(2,d_4,4)\ne\emptyset$. 
\end{itemize}
\end{theorem}
\begin{proof}
This follows at once from Theorem \ref{main} and Proposition \ref{prop5}.
\end{proof}

\begin{theorem}\label{cs4} Let $C$ be a general curve of genus $10$. Then
\begin{itemize}
\item[(i)]  $U(2,d,4)\ne\emptyset$ if $d\ge13$; 
\item[(ii)]  $G(\alpha; 2,12,4)=\emptyset$ for all $\alpha>0$.
\end{itemize}
However
\begin{itemize}
\item[(iii)] there exist Petri curves of genus $10$ for which $U(2,12,4)\ne\emptyset$.
\end{itemize}
\end{theorem}
\begin{proof}
(i) is a special case of Theorem \ref{cs2}(ii). 

(ii) and (iii) follow from Theorem \ref{main2} and Proposition \ref{prop5}.
\end{proof}

\section{Postscript}\label{ps}

In this postscript, added in March 2011, we comment on some remarkable developments which have taken place since the body of the paper was completed in August 2010.

These developments concern the construction of bundles providing counter-examples to Mercat's conjecture and relating them to Koszul cohomology, the maximal rank conjecture and the geometry of the moduli space of curves \cite{fo,ln2,ln3,fo2}. In particular many examples of bundles of rank $2$ have been constructed which contradict Mercat's conjecture, all of which involve the non-injectivity of the map $\psi$ of \eqref{eqnpsi}. All the curves involved lie on K3 surfaces and it remains possible that $\Clt=\Cl$ for the general curve of any genus. In particular, for a general curve of genus $g=11$, it is proved in \cite[Theorem 5.1]{fo2} that $\Clt=\Cl$; on the other hand, there exist curves of genus $11$ for which $\Cl$ takes its maximal value $5$ with $\Clt<\Cl$ \cite[Theorem 1.4]{fo}. This does not completely answer the question raised in Remark \ref{r11} since we do not know whether any of the latter curves can be Petri.

For a general curve of any odd genus, it is proved in \cite[Theorem 1.1]{fo} that  $B(2,g+3,4)$ is non-empty. Thus the result described in Remark \ref{rmk1} can be improved to state that, for a general curve of any genus, $B(2,g+3,4)\ne\emptyset$ if $d\ge g+3$. Note here that the condition $d\ge g+3$ is equivalent to $\beta(2,d,4)\ge0$. For $g=7$ and $g=9$, we have $d_4=g+3$, so in these cases
$$B(2,d,4)\ne\emptyset\Longleftrightarrow d\ge g+3,$$
improving the result of Proposition \ref{prop7}. In fact, \cite[Theorem 1.3]{fo} allows us to go further. Again, for a general curve of any odd genus, the locus in $B(1,g+3,5)$ consisting of those $M$ for which \eqref{eqnpsi} is not injective has a component of dimension $2$ whose general point $M$ is not of the form $L\otimes L'$ with $L, L'\in B(1,d_1,2)$ (note that $d_1=\frac{g+3}2$). By the last paragraph of Remark \ref{r5}, this gives a component of dimension $2$ of $B(2,g+3,4)$. This resolves the issues raised in Remark \ref{r10}.

In the case $g=7$, a stronger result can be proved very simply for any Petri curve, using the fact that all classical Brill-Noether loci of positive dimension are irreducible. Consider the irreducible variety $V=B(1,5,2)\times B(1,5,2)$ of dimension $2$. Note that, if $(L,L)\in V$, then $h^0(L\otimes L)=4$ (see \cite[Lemma 2.10]{ln3} for a proof). It follows by semicontinuity that $h^0(L\otimes L')\le4$ (hence $=4$ by the base-point free pencil trick) for the general point $(L,L')$ of $V$. Since $B(1,10,5)$ also has dimension $2$, it follows that the general $M\in B(1,10,5)$ is not of the form $L\otimes L'$ for $(L,L')\in B(1,5,2)$.  It follows then from Remarks \ref{r5} and \ref{r10} and the irreducibility of $B(1,10,5)$ that $B(2,10,4)$ is irreducible of dimension $2$.

Finally, we are grateful to Gavril Farkas for demonstrating to us that  the answer to the question raised at the end of section \ref{g=10} is the expected one, namely that on a general curve $C$ of genus $10$, there do not exist line bundles $M, M'\in B(1,12,5)$ such that the multiplication map (12) is not surjective. This implies, by Proposition \ref{prop6}, that there are no stable bundles on $C$ computing $\Clt$. The proof runs as follows. Let ${\mathcal M}$ denote the moduli space of curves of genus $10$ and let ${\mathcal G}^1_{10,6}$ be the Hurwitz scheme parameterising pairs $(C,L)$, where $C$ is a curve of genus $10$ and $L\in B(1,6,2)$. The fibre product $X:={\mathcal G}^1_{10,6}\times_{\mathcal M}{\mathcal G}^1_{10,6}$ then parameterises triples $(C,L,L')$ with $L,L'\in B(1,6,2)$ or equivalently, by duality, triples $(C,M,M')$ with $M,M'\in B(1,12,5)$. The scheme $X\setminus \Delta$, where $\Delta$ denotes the diagonal, parametrises morphisms $C\to{\mathbb P}^1\times{\mathbb P}^1$ of bidegree $(6,6)$ which are birational onto their image. It follows that $X\setminus\Delta$ is irreducible. Now suppose that \eqref{eqn11} fails to be surjective for some $M\ne M'$ on every curve $C$ of genus $10$. Then, since $X\setminus\Delta$ is irreducible and the generic fibre of $X\to{\mathcal M}$ is finite, \eqref{eqn11} fails to be surjective for every  pair $M\ne  M'$; hence \eqref{eqnpsi} fails to be surjective for some $M$ on any curve in the branch locus of ${\mathcal G}^1_{10,6}$. This branch locus coincides with  the Gieseker-Petri locus $GP^1_{10,6}$, which is known to be a divisor in ${\mathcal M}$, while the locus of curves which admit a line bundle $M$ with \eqref{eqnpsi} not surjective is also a divisor, which coincides with the irreducible divisor ${\mathcal K}_{10}$ of curves which lie on K3 surfaces. So we have $GP^1_{10,6}\subset {\mathcal K}_{10}$ and these two divisors must coincide set-theoretically. This is a contradiction since the slopes of these divisors are known (see \cite{eh} for $GP^1_{10,6}$ (where it is denoted by $E^1_6$) and \cite{fp} for ${\mathcal K}_{10}$) and are different (in fact ${\mathcal K}_{10}$ is the unique divisor in ${\mathcal M}$ having slope contradicting the slope conjecture of Harris and Morrison). This completes the proof and demonstrates again the close connection between higher rank Brill-Noether theory and the geometry of the moduli space of curves.

\end{document}